\documentclass[a4paper,twoside,12pt]{article}
\usepackage{amssymb,amsfonts,amsmath,amsthm,latexsym}
\usepackage[pdftex,bookmarks,colorlinks=false]{hyperref}
\usepackage[hmargin=1.2in,vmargin=1.2in]{geometry}
\usepackage{graphicx}
\usepackage[auth-sc-lg,affil-sl]{authblk}
\usepackage{tabularx}

\newtheorem{thm}{Theorem}[section]
\newtheorem{cor}[thm]{Corollary}
\newtheorem{defn}[thm]{Definition}
\newtheorem{lem}[thm]{Lemma}
\newtheorem{prob}[thm]{Problem}
\newtheorem{prop}[thm]{Proposition}

\newtheorem{ill}[thm]{Illustration}

\numberwithin{equation}{section}
\def\ni{\noindent}
\def\N{\mathbb{N}}

\def\cS{\mathcal{S}}

\title{\textbf{\sc On the Curling Number of Certain Graphs}}

\author{Johan Kok}
\affil{\small Tshwane Metropolitan Police Department \\ City of Tshwane, Republic of South Africa. \\ E-mail: kokkiek2@tshwane.gov.za}

\author{Sudev Naduvath}
\affil{\small Department of Mathematics \\ Vidya Academy of Science \& Technology \\Thalakkottukara, Thrissur, India.\\ E-mail: sudevnk@gmail.com}

\author{Chithra Sudev}
\affil{\small Naduvath Mana, Nandikkara \\ Thrissur, India.\\ E-mail: chithrasudev@gmail.com}

\date{}

\begin{document}
\maketitle

\begin{abstract}
In this paper, we introduce the concept of curling subsequence of simple, finite and connected graphs. A curling subsequence is a maximal subsequence $C$ of the degree sequence of a simple connected graph $G$ for which the curling number $cn(C)$ corresponds to the curling number of the degree sequence per se and hence we call it the curling number of the graph $G$. A maximal degree subsequence with equal entries is called an identity subsequence.  The number of identity curling subsequences in a simple connected graph $G$ is denoted $ic(G)$. We show that the curling number conjecture holds for the degree sequence of a simple connected graph $G$ on $n \ge 1$ vertices. We also introduce the notion of the compound curling number of a simple connected graph $G$ and then initiate a study on the curling number of certain standard graphs like Jaco graphs and set-graphs. 
\end{abstract}

\ni {\bf Keywords:} Curling number, curling subsequence, identity curling subsequence, compound curling number, Rasta graph.

\ni {\footnotesize \textbf{AMS Classification Numbers:} 05C07, 05C20, 05C38, 05C70, 05C75}
 
\section{Introduction}

For general notation and concepts in graph theory, we refer to \cite{BM1,CL1,GY1} and\cite{DBW}. All graphs mentioned in this paper are simple, connected finite graphs unless mentioned otherwise. 

The notion of \textit{curling number} of a sequence of integers is introduced in \cite{CLSW} as follows. 

Let $S=S_1S_2S_3\ldots S_n$ be a finite string. Write $S$ in the form $XYY\ldots Y=XY^k$, consisting of a prefix $X$ (which may be empty), followed by $k$ copies of a non-empty string $Y$. This can be done in several ways. Pick one with the greatest value of $k$ . Then, this integer $k$ is called the {\em curling number} of $S$ and is denoted by $cn(S)$. 

The \textit{Curling Number Conjecture} (see \cite{CS1}) states that if one starts with any finite string, over any alphabet, and repeatedly extends it by appending the curling number of the current string, then eventually one must reach a $1$.

More studies on curling number of integer sequences have been done in \cite{CLSW} and and \cite{CS1}. Motivated from these studies, we extend the concepts of curling number of sequences to the degree sequences of graphs and hence establish some interesting results in this area. 

In our present discussion, we slightly amend our notation and definition of the curling number so that we are able to introduce a new notion called \textit{compound curling number} of a simple connected graph $G$, as explained below. 
 
\begin{defn}{\rm 
Given an initial finite non-empty sequence $S_0 = (a_1, a_2, a_3, \ldots, a_n)$, $a_i \in \N_0$, let $\epsilon$ be the empty subsequence. If we partition the sequence $S_0$ into two subsequences, say $X,Y$, we can write $S_0$ as the string $S_0 =X\circ Y$. Hence, we can write sequence $S_0$ as a string of subsequences $S_0 =\epsilon \circ X_1^{k^*_1}$ or $S = X_1^{k_1}\circ X_2^{k_2}\circ \ldots \circ X_l^{k_l}$, where $X_i^{k_i} \ne \epsilon$, $1 \le i \le l, ~ l \ge 2$.}
\end{defn}

Henceforth, the initial sequence $S_0$ will be considered to be finite and non-empty unless mentioned otherwise.

\begin{defn}\label{Def-1.2}{\rm 
If after re-arrangement of the entries of a general sequence $S$ to obtain a specific initial sequence $S_0$, we can write $S_0 = \epsilon \circ  X_1^{k^*_1}$ or $S = X_1^{k_1}\circ X_2^{k_2}\circ \ldots\circ X_l^{k_l}$, where $X_i^{k_i} \ne \epsilon$, $1 \le i \le l, l \ge 2$ and $k_1 \le k_2 \le k_3 \le \ldots \le k_l$ and $k_l$ a maximum over all possible subsequences, we define the curling number of $cn(S_0) = \max\{k^*_1, k_l\}$.}
\end{defn}

The term $X^{k_i}_i$, explained in Definition \ref{Def-1.2}, is called a {\em string factor} of the initial sequence $S_0$. 

Note that if $S_0$ has $n$ entries and is written $S_0 = \epsilon \circ  X^{k^*_1}_1$ then $k^*_1 \in \{1, n\}$. Also note that for the purposes of determining the curling number of an initial sequence $S_0$ we only have to write either $S_0 = \epsilon \circ X^{k^*_1}_1$ or $S_0 = X^{k_1=1}_1\circ X^{k_2}_2$ for $\max k_2$ over all possible subsequences. It follows easily that $cn(S_0) = cn(X^{k_2}_2)$ for $\max k_2$ over all possible subsequences.

\section{Curling Sub-Sequences of Graphs}

Let us first introduce the notion of a curling subsequence of the degree sequence of a simple connected graph as follows.

\begin{defn}{\rm 
A \textit{curling subsequence} of a simple connected graph $G$ is defined to be a maximal subsequence $C$ of the well-arranged degree sequence of $G$ such that $cn(C) = \max\{cn(S_0)\}$ for all possible initial subsequences $S_0$. } 
\end{defn}

\begin{defn}{\rm 
The \textit{curling number} of a graph $G$ is the curling number of a curling subsequence $C$ of $G$. That is, $cn(G) = cn(C)$, where $C$ is a curling subsequence of $G$.}
\end{defn}

Note that a graph $G$ can have a number of curling subsequences. Hence, we have the following notions.

\begin{defn}{\rm 
The number of curling subsequences in a given graph $G$ is called the \textit{curling index} of $G$ and is denoted by $\vartheta(G)$.}
\end{defn} 

\begin{defn}{\rm 
A maximal degree subsequence with equal entries is called an {\em identity subsequence}. An identity subsequence can be a curling subsequence, the number of identity curling subsequences found in a simple connected graph $G$ is denoted $ic(G)$.}
\end{defn}
 
\ni An illustration to the above notions is given below.

\begin{ill}{\rm 
Consider the graph $G$ with degree sequence $(3,5,3,3,5,5,6)$. After re-arrangement of the entries, we can consider the four possible initial degree sequences $S_0=(6,3,3,3,5,5,5)$ or $S_0=(6, 5,5,5,3,3,3)$ or $S_0=(6,3,5,3,5,3,5)$ or $S_0=(6,5,3,5,3,5,3)$. We can identify the two identity curling subsequences $(3,3,3)$ and $(5,5,5)$ as well as the two additional curling subsequences $(3,5,3,5,3,5)$ and $(5,3,5,3,5,3)$. Since $(3,3,3)=\epsilon (3)^3$,  we have $cn((3,3,3))=3$. Similarly, $(5,5,5) = \epsilon (5)^3, cn((5,5,5)) =3$, $(3,5,3,5,3,5) = \epsilon (3,5)^3, cn((3,5,3,5,3,5)) =3$ and $(5,3,5,3,5,3) = \epsilon (5,3)^3, cn((5,3,5,3,5,3))= 3$. Therefore $cn(G) =3$. Note that $(6)$ is an identity subsequence as well.}
\end{ill}

Invoking the above definitions, we establish the following theorem on the number of curling subsequences of a given graph.

\begin{thm}\label{Thm-2.2}
The number of curling subsequences of a simple connected graph $G$ is given by 
\begin{equation*}
\vartheta(G) =
\begin{cases}
1 & \text{if} ~ ic(G)=1\\
ic(G) +  ic(G)! & \text{otherwise}
\end{cases}
\end{equation*}
\end{thm}
\begin{proof}
If $ic(G) =1$, then, clearly by the definition of curling subsequence, $G$ can have exactly one curling subsequence. That is, $\vartheta(G) = 1$ if $ic(G) =1$ 

Then, for a positive integer $l \ge 2$, assume that the graph $G$ has $l$ identity curling subsequences, say $S_1 = (\alpha _1, \alpha _1,\alpha _1,\ldots,\alpha _1)$,  $S_2 = (\alpha _2,\alpha _2,\alpha _2,\ldots,\alpha _2,)$, $S_3 = (\alpha _3, \alpha _3,\alpha _3,\ldots, \alpha _3)$ and proceeding like this up to $S_l = (\alpha _l,\alpha _l, \alpha _l,\ldots, \alpha _l)$, each of these subsequences containing $k \ge 2$ terms. 

Now, we can construct exactly $ic(G)\,! = l\, !$ sequences of the form $(\alpha _i,\alpha _j, \alpha _m, \ldots,\\ \alpha _s)$ containing $l$ entries each of which are distinct. Hence, the degree sequence can be written as $X^{k_1 = 1}_1\circ X^{ic(G)}_2$, with $ X^{ic(G)}_2 \in \{\alpha _i,\alpha _j, \alpha _m, \ldots, \alpha _s\}$, where each of these $l$ entries are distinct and $1\le i,j,m,\ldots,s\le l$. Therefore,  $\vartheta(G) =  ic(G) + ic(G)!$ ways. This completes the proof.
\end{proof}

The validity of curling number conjecture for a simple, connected graph under consideration is one of the most interesting problem in this context. The following theorem establishes the existence of curling number conjecture for a given graph. 

\begin{thm}\label{Thm-2.3}
For the degree sequence of a non-trivial, connected graph $G$ on $n$ vertices, the curling number conjecture holds.
\end{thm}
\begin{proof}
Consider a graph $G$ on $n$ vertices. If $n=1$, then $G=K_1$ and hence $S_0=(0)$, $S_1=(0, 1)$ so the conjecture holds trivially. Now, let $n \ge 2$.  If the degree sequence of $G$ can only be written as $S_0 = \epsilon \circ X^{k^*_1}_1$ then $k^*_1 =$ either $1$ or $n$.

\vspace{0.2cm}

\ni \textit{Case 1:} If $k^*_1 = 1$ and the last entry of the degree sequence is 1, then $S_1 = S_0 \circ (1)$, $S_2 = S_0\circ (1,2)$, $S_3 = S_0\circ (1,2,1)$.  If the last entry is larger than 1, then $S_1 = S_0 \circ (1)$.  

\vspace{0.2cm}

\ni \textit{Case 2:} If $k^*_1 = n$ the $G$ is a regular graph. So $S_0 =\epsilon \circ X^n_2$. So $S_1 = \epsilon \circ X^n_2 \circ (n)$, $S_2 = \epsilon \circ X^n_2 \circ (n 1)$. 

\vspace{0.2cm}

\ni \textit{Case 3:} If the degree sequence can be written as $X^{k_1 = 1}_1 \circ X^{k_2}_2$ for $\max k_2$ and $X^{k_2}_2$ is an identity curling subsequence, over all possible subsequences, we have that if the entry values are $\alpha$ and the number of entries is $\alpha$, then $S_1 = X^{k_1 = 1}_1 \circ X^\alpha_2 \circ (\alpha)$, $S_2 = X^{k_1 = 1}_1 \circ X^\alpha_2 \circ (\alpha, \alpha + 1)$, $S_3 = X^{k_1 = 1}_1 \circ X^\alpha_2 \circ (\alpha, \alpha + 1, 1)$. For all other cases we have that $S_0 = X^{k_1 = 1}_1 \circ X^{k_2}_2 = X^{k_1 = 1}_1 \circ X^{\beta}_2, \beta \ne \alpha$. So $S_1 = X^{k_1 = 1}_1 \circ X^{\beta}_2 \circ (\beta)$, $ S_2 = X^{k_1 = 1}_1 \circ X^{\beta}_2 \circ (\beta,1)$. 

Hence, the conjecture holds in all the above three cases. This completes the proof.
\end{proof}

We note that any degree sequence can be re-arranged to be written as a string of identity subsequences. In view of this fact, we propose the following definition.

\begin{defn}{\rm 
If the degree sequences of graphs $G_1$ and $G_2$ are written as strings of identity subsequences and the entries in two identity subsequences $X_i, X_j$ ($X_i$ in $G_1$ with $s$ entries and $X_j$ in $G_2$ with $t$ entries) are equal they are called \textit{similar identity subsequences} and the merger of these identity subsequences is the identity subsequence with $s+t$ entries which is denoted, $X_i\uplus X_j$.}
\end{defn}

The merger of similar identity subsequences plays an important role when considering the union of a number of graphs. Hence, if $G = \bigcup\limits^{m}_{i=1}G_i$ and each $G_i$ has a similar identity subsequence say $X_i=(\alpha, \alpha, \alpha, \ldots, \alpha)$ containing $k_i$ entries each, then $X_1\uplus X_2\uplus X_3\uplus \ldots \uplus X_m = (\alpha)^k$, where $k=\sum\limits^{m}_{i=1}k_i$. 

If a particular degree sequence of $G_j$ does not contain the identity subsequence $X_r$ then it is represented by the empty identity subsequence $\epsilon$ and $k_r = 0$.

Invoking the above concepts on similar identity subsequences, we establish the following theorem on the curling number of a graph which is the union of finite number of finite graphs. 

\begin{thm}
If a graph $G$ is the union of $m$ simple connected graphs $G_i;~1\le i\le m$ and the respective degree sequences are re-arranged as strings of identity subsequences, then
\begin{equation*} 
cn(G)=
\begin{cases}
\max\{cn(G_i)\}, & \text{if} ~ X_i, X_j \text{are not pairwise similar},\\ 
\max\sum\limits^{m}_{i=1}k_i, & \text {for all merger of similar identity subsequences.}
\end{cases}
\end{equation*} 
\end{thm}
\begin{proof}
\begin{enumerate}
\item[(i)] If $X_i, X_j$ are not pairwise similar for any values of $i$ and $j$, then any identity subsequence can be written as $(\alpha)^{k_i\ge 1}$. Hence, a $\max k_j$ exists for some $j$ thus, $cn(G) = cn(\bigcup\limits^{m}_{i=1}G_i) = \max\{cn(G_i): 1 \le i \le m\}$.

\item[(ii)] After all mergers of similar identity subsequences, re-arrange the degree sequence of $G$ in such a way that it can be written as $(\alpha_1)^{k_1}\circ (\alpha_2)^{k_2}\circ )\alpha_3)^{k_3}\circ \ldots \circ (\alpha_l)^{l = \max\sum\limits^{m}_{i=1}k_i}$. From the Definition \ref{Def-1.2}, it follows that, $cn(G) = cn(\bigcup\limits^{m}_{i=1}G_i) = \max\sum\limits^{m}_{i=1}k_i$.
\end{enumerate}
\end{proof}

\subsection{Compound Curling Number of a Graph}

As we have stated earlier, any degree sequence of a graph $G$ can be written as a string of identity curling subsequences. In view of this fact, we introduce the concept of the compound curling number of a graph $G$ as follows.

\begin{defn}\label{Def-3.1}{\rm 
Let the degree sequence of the graph $G$ be written as a string of identity curling subsequences say, $X^{k_1}_1\circ X^{k_2}_2\circ X^{k_3}_3 \ldots \circ X^{k_l}_l$. The \textit{compound curling number} of $G$, denoted $cn^c(G)$ is defined to be, $cn^n(G) = \prod\limits^{l}_{i=1}k_i$.}
\end{defn}

\ni Next, let us illustrate the curling number and compound curling number of certain fundamental standard graphs. 

\begin{enumerate}\itemsep0mm
\item The degree sequence of $K_n$ is $(n-1,n-1,n-1,\ldots,n-1)$, consisting of $n$ terms. Hence, $S_0=\epsilon\circ (n-1)^n$. Therefore, $cn(K_n)=n$ and $cn^c(K_n)=n$.

\item The degree sequence of $K_{m,n}$ is $(\underbrace{m,m,\ldots,m,}_{n-terms}\underbrace{n,n,\ldots,n}_{m-terms})$. Therefore, $S_0=(m)^n\circ (n)^m$. Hence, $cn(K_{m,n})=\max\{m,n\}$. Similarly, for $K_{n,n}$, we have $S_0=\epsilon\circ (n)^{2n}$ and hence $cn(K_{n,n})=2n$. Moreover, the compound curling number of $K_{m,n}$ is $mn$ and that of $K_{n,n}$ is $2n$ itself.

\item For a path $P_n$ on $n$ vertices, we have $S_0=(1)^2\circ (2)^{n-2}$. If $n=2$, then $S_0=\epsilon\circ (1)^2$ and if $n=3$, then $S_0= (1)^2\circ (2)^1$. Hence, $cn(P_{n})=2$ for $n=2,3$ and $cn(P_{n})=n-2$ if $n\ge 4$. Also, the compound curling number of $P_n$ is given by $cn^c(P_{n})=2$ for $n=2,3$ and $cn^c(P_n) = 2(n-2)$.
 
\item For a cycle $C_n$, we have $S_0=\epsilon\circ(2)^n$ and hence $cn(C_n)=n=cn^c(C_n)$.

\item A \textit{wheel graph} is defined as $W_n=C_{n-1}+K_1$. The degree sequence of $W_n$ is $S_0=(n-1)^1\circ (3)^{n-1}$. Hence, the curling number and the compound curling number of the wheel graph $W_n$ are $n-1$. 

\item A ladder graph $L_n$ is the graph defined by $L_n=P_n\times P_2$, where $n\ge 2$. If $n=2$, then $L_2=C_4$, and this case we have already discussed. Therefore, for $n>2$, in $L_n$, there are $2n$ vertices of which $4$ vertices have degree $2$ and all other vertices have degree $3$. That is, the degree sequence of $L_n$ is $(2)^4\circ (3)^{2n-4}$. Therefore,  $cn(L_n) = 2(n-2)$ and $cn^c(L_n) = 8(n-2)$.

\item An $m$-regular graph of girth $n$ with minimum possible vertices is called an \textit{$(m,n)$-cage}.  $(2,n)$-cage is an $n$-cycle and hence its curling number is $n$, $(m,3)$-cage is a complete graph $K_{m+1}$ and hence its curling number is $m+1$ and the $(m,4)$-cage is the complete bipartite graph $K_{m,m}$ and hence its curling number is $2m$. 
For a general (m,n)-cage $G$, the number of vertices required is denoted by $f(m,n)$ and is determined as
$f(m,n) =
\begin{cases}
\frac{m(m-1)^r - 2}{m-2}, & \text {if}~ r = \frac{n-1}{2},\\
\frac{2(m-1)^r - 2}{m-2}, & \text {if}~ r = \frac{n}{2}.
\end{cases}$ 
Therefore, the curling number of an $(m,n)$-cage is $f(m,n)$. It can also be noted that the compound curling number of a cage graph is equal to its curling number. 
\end{enumerate}

\ni Invoking the above definitions, we arrive at the following results.

\begin{prop}
The compound curling number of any regular graph $G$ is equal to its curling number.
\end{prop}
\begin{proof}
Let $G$ be an $r$-regular graph on $n$ vertices. Then, the proof is obvious from the fact that the degree sequence of $G$ is $\epsilon\circ (r)^n$.
\end{proof}

\begin{prop}
If the degree sequence of two graphs $G$ and $H$ can be written as a string of identity curling subsequences, $X^{k_1}_1\circ X^{k_2}_2\circ X^{k_3}_3\circ \ldots \circ X^{k_t}_t$, and $Y^{l_1}_1\circ Y^{l_2}_2\circ Y^{l_3}_3\circ \ldots \circ X^{k_i}_i\circ \ldots \circ Y^{l_s}_s$, respectively, then the compounding curling number of the union of $G$ and $H$ is given by $cn^c(G\cup H) = cn^c(G) \prod \limits_{1\le j\le s,~ j\ne i}Y^{l_j}_{j} +cn^c(H) \prod \limits_{1\le j\le t, ~ j\ne i}X^{k_j}_j$.
\end{prop}
\begin{proof}
The proof of this proposition is very clear and straight forward.
\end{proof}

Let $n=\sum\limits_{i=n}^lt_i$, where each $t_i \in \N$ such that $t_i>t_{i+1}, ~ \forall ~i=1,2,\ldots, l$ and $t_l>1$. The general \textit{$l$-term summand set} is the set $\{t_1, t_2, t_3, \ldots, t_l\}$. 

By a {\em minimal graph} with respect to a given property, we mean a graph with minimum order and size satisfying that property. Next, using this terminology, let us introduce the terminology of some connected graphical embodiment for general {\em $l$-term} summand sets.

We now describe the construction of the minimal connected graphical embodiment called the {\em Rasta graph}, denoted by $G^{(l)}$, for a general $l$-term summand set which satisfies the condition $cn^c(G^{(l)}) = \prod\limits_{1\le i\le l, ~ i\ne 3}t_i$, and $cn^c(G^{(3)}) = t_2(t_1 + t_3)$ as follows.

\subsection*{Constructing a Rasta graph} 

\ni Consider a $t$-term summand set $\{t_1, t_2, t_3, \ldots, t_l\}$ with $t_1 > t_2 > t_3 > \ldots > t_l>1$. Then, a Rasta graph can be constructed in the following steps.
\begin{enumerate}\itemsep0mm
\item[S-1:] Consider $t_1$ vertices in the left column (first column) and $t_2$ vertices in the right column (second column) and construct $K_{t_1, t_2}$.
\item[S-2:] Add the third column of $t_3$ vertices and add the edges of $K_{t_2, t_3}$
\item[S-3:] Repeat Step 2 iteratively up to $t_l$
\item[S-4:] Exit.
\end{enumerate}

We note that a Rasta graph is the underlying graph of a directed graph mentioned in the following definition.

\begin{defn}{\rm 
For a $l$-term summand set $\{t_1, t_2, t_3, \ldots, t_l\}$ with $t_1 > t_2 > t_3 > \ldots > t_l>1$ define the directed graph $G^{(l)}$ with vertices $V(G^{(l)}) =\{v_{i,j}: 1\le j \le t_i, 1\le i \le l \}$ and the arcs, $A(G^{(l)}) =\{(v_{i,j},v_{(i+1), m}):1\le i \le (l-1), 1\le j \le t_i$ and $1 \le m \le t_{(i+1)}\}$. }
\end{defn}

The following lemma establishes a relation between the compound curling numbers of a complete bipartite graph and a connected graphical embodiment of a $2$-term summand set.

\begin{lem}\label{Lem-2.14}
For a $2$-term summand set $\{a,b\}$, where $a > b$, compound curling number of the minimal connected graphical embodiment $G^{(2)}$ is given by $cn^c(G^{(2)})$ is equal to that of a complete bipartite graph $K_{a,b}$.
\end{lem}
\begin{proof}
Clearly $G^{(2)}$ must have $a+b$ vertices. Connecting each vertex $v_{1,j}$, where $ 1 \le j \le a$ to each vertex $v_{1,m}$ $1\le m \le b$ ensures the minimal connected graphical embodiment $G^{(2)}$ such that $d_G^{(2)}(v_{1,j}) = b$, $1 \le j \le a$. Hence, a factor of the degree sequence string is given by $(b)^a$. It follows that this minimal graphical embodiment resulted in the graph $G^{(2)} = K_{a,b}$. Hence, the degree sequence can be written as the string $(b)^a\circ (a)^b$. Therefore, $cn^c(G^{(2)})=ab=cn^c(K_{a,b})$.
\end{proof}

The compound curling number of the Rasta graph corresponding to an $l$-term summand set is determined in the following theorem. 

\begin{thm}
For $n \in \N$ and any $l$-term summand set of $n$ say $\{t_1, t_2, t_3, \ldots, t_l\}$ with $t_1 > t_2 > t_3 > \ldots > t_l>1$,  the corresponding Rasta graph is a minimal connected graphical embodiment $G^{(l)}$ such that $cn^c(G^{(l)}) = \prod\limits_{i=1,\,i\ne 3}^{l}t_i$ and $cn^c(G^{(3)}) = t_2(t_1 + t_3)$.
\end{thm}
\begin{proof}
For a $2$-term summand set $\{t_1,t_2\}$, where $t_1 > t_2$, the Rasta graph is $K_{t_1,t_2}$. Clearly, the number of vertices and edges are minimal to provide the degree sequence $(\underbrace{t_2, t_2, \ldots, t_2}_{t_1-entries}, \underbrace{t_1, t_1, t_1, \ldots, t_1}_{t_2-entries})$ (See Lemma \ref{Lem-2.14}). Hence, the degree sequence can be written as the string $(t_2)^{t_1}\circ (t_1)^{t_2}$ so $cn^c(K_{t_1,t_2}) = \prod\limits_{i=1}^{2}t_i$.

For a $3$-element summand set $\{t_1,t_2,t_3\}$, the corresponding Rasta graph has degree sequence  can be written as the string of identity subsequences $S_0 = (t_2)^{t_1}\circ (t_1+t_3)^{t_2}\circ (t_2)^{t_3} = (t_1+t_3)^{t_2}\circ (t_2)^{(t_1+t_3)}$. Therefore, $cn^c(G^{(3)}) = t_2(t_1 + t_3)$.

Finally, consider an general $l$-term summand set $\{t_1, t_2, t_3, \ldots, t_l\}$ of an integer $n$, where ${l \ne 3}$ such that $t_1 > t_2 > t_3 > \ldots > t_l >1$ and $l \ge 2$. Clearly, the  Rasta graph corresponding to this $l$-term summand set is a minimal graph whose degree sequence is 
$(t_2)^{t_1}\circ (t_1+t_3)^{t_2}\circ (t_2+t_4)^{t_3} \circ \ldots,\circ (t_{l-2}+t_l)^{t_{(l -1)}}\circ (t_{l -1})^{t_l}$. Also, we have $t_1 > t_2 > t_3 > \ldots > t_l>1$. 
Therefore, the result $cn^c(G^{(l)}) = \prod\limits_{i=1,\,i \ne 3}^{l}t_i$, follows.
\end{proof}

Let us now introduce the notion of the maximal $l$-term summand set in respect of $n \in \N$ as follows.

\begin{defn}{\rm 
Let $n\ge $ be a positive integer such that $n=\sum\limits_{i=1}^{l}t_i, ~ \forall\, t_i \in \N$. The \textit{maximal $l$-term summand set} of $n\ge 3$ is the $l$-term summand set $\{t_1,t_2,t_3,\ldots t_l\}$ such that $t_{i+1}>t_i$ for $i=1,2,3,\ldots,l$ for which the product $\prod\limits_{i=1}^{l}t_i$ is maximum.}
\end{defn}

A maximal $2$-term summand set for a positive integer $n$ is determined in the following lemma.

\begin{lem}\label{Lem-2.17}
Let $n$ be a positive integer greater than or equal to $3$. Then, the maximal $2$-term summand set for $n$ is $\{\lfloor\frac{n}{2}\rfloor + 1, \lfloor\frac{n}{2}\rfloor\}$.
\end{lem}
\begin{proof}
By induction, it can easily be followed that, $\lfloor\frac{n}{2}\rfloor( \lfloor\frac{n}{2}\rfloor + 1) > (\lfloor\frac{n}{2}\rfloor -i)((\lfloor\frac{n}{2}\rfloor +1) + i) = \lfloor\frac{n}{2}\rfloor( \lfloor\frac{n}{2}\rfloor + 1) - i - i^2$ and this completes the proof.
\end{proof}

We observe that, for a sufficiently large $n$, a maximal $3$-term summand set $\{t'_1, t'_2, t'_3\}$, such that $t'_1>t'_2>t'_3$, can be obtained from a maximal $2$-term summand set by applying Lemma \ref{Lem-2.17} to $t_1$. If $n$ is sufficiently large, the maximal $4$-term summand set can be obtained by applying Lemma \ref{Lem-2.17}  to $t'_i = \max\{t'_1,t'_2,t'_3\}$ for which a maximal $4$-term summand set is defined and the process can be repeated iteratively.
 
\ni In view of the above concept, we have the following result.
 
\begin{cor}
There exists a maximum value $l^* \ge 4$ for which we find a defined $l$-term summand set with respect to the given positive integer $n\ge 3$, has corresponding Rasta graph $G^{(l^*)}$ with $cn^c(G^{(l^*)})$ is maximum over all defined $l$-term summand sets.
\end{cor}
\begin{proof}
Consider any $l$-term summand set with respect to a positive integer $n\ge 3$, say $\{t_1, t_2, t_3, \ldots, t_l\}$. Apply Lemma \ref{Lem-2.17} to $t'_j = \max\{t_1,t_2,t_3, \ldots, t_l\}$ for which a maximal ($l+1$)-term summand set is defined. Now, we have $cn^c(G^{(l)}) = \prod\limits_{i=1,\,i\ne 3}^{l}t_i = (\prod\limits_{i=1,\,i\ne 3,j}^{l}t_i) t_j = (\prod\limits_{i=1,\,i\ne 3,j}^{l}t_i)(a+b)$, where $\{a,b\}$ is the maximal $2$-term summand set of $t_j$. Clearly, $(\prod\limits_{i=1,\,i\ne 3,j}^{l}t_i)(a+b)<(\prod\limits_{i=1,\,i\ne 3,j}^{l}t_i)(ab)$.  
If a ($l$ + 1)-term summand set is not defined, then the set $\{t_1, t_2, t_3, \ldots, t_l\}$ gives the required result. Repeat same procedure until the final $l^*$-term summand set is found through exhaustion.
\end{proof}

An illustration to maximal summand sets for a given positive integer $n$ and the compound curling number of the corresponding Rasta graphs is given below.

\begin{ill}{\rm 
For $n = 30$, the maximal $2$-term summand set is $\{16,14\}$ and the maximal $3$-term summand set is $\{14,9,7\}$, the maximal $4$-term summand set is $\{9,8,7,6\}$, the maximal $5$-term summand set is $\{8,7,6,5,4\}$ and the maximal $6$-term summand set is $\{8,7,6,4,3,2\}$. No other maximal $l$-term summand sets are defined for $n=30$. The  degree sequence of corresponding Rasta graph $G^{(6)}$ can be written as the string $S_0=(7)^8\circ(14)^7\circ (11)^6\circ (9)^4\circ (6)^3\circ (3)^2$. Hence, we have $cn^c(G^{(6)}) = 8064$, $cn^c(G^{(5)}) = 6720$, $cn^c(G^{(4)} = 3024$, $cn^c(G^{(3)}) = 189$ and $cn^n(G^{(2)}) = 224$.} 
\end{ill}

\begin{cor}
The compound curling number of the Rasta graph $G^{(3)}$ corresponding to a $3$-term summand set is the minimum among the compound curling numbers of the Rasta graphs $G^{(l)}$ corresponding to all $l$-term summand sets of $n$.
\end{cor}
\begin{proof}
We only have to show that $cn^n(G^{(2)}) > cn^n(G^{(3)})$. For a maximal $2$-term summand set $\{a,b\}$ we have $a-b \le 2$. Hence if $\{a_1,a_2\}$ is the maximal $2$-term summand set of $a$, the maximal $3$-term summand set is given by $\{b,a_1,a_2\}$. We have that $cn^c(G^{(2)}) = a b = (a_1+a_2) b = a_1b + a_2b$. We also have that $cn^c(G^{(3)}) = a_1 (b + a_2) = a_1b + a_1a_2 < a_1b + a_2b$ because $b > a_1$.
\end{proof}

In general, there exists no specific relation between the compound curling number of a graph and its spanning subgraphs. We discuss a relation between the compound curling numbers of a regular graph and its spanning subgraph in the following result.

\begin{prop}\label{Prop-2.21}
For a $k$-regular graph G on $n\ge 4$ vertices we have $cn^c(G - uv) \ge cn^c(G)$. 
\end{prop}
\begin{proof}
For a $k$-regular graph $G$ on $n \ge 4$ vertices, $cn^c(G) = n$. Consider any $G-uv$ and it follows that the corresponding degree sequence can be written as the string $(k)^{n-2}\circ (k-1)^2$. Hence, $cn^c(G-uv)= 2(n-2)=n+(n-4)\ge n$. Hence, $cn^c(G - uv) \ge cn^c(G)$.
\end{proof}

Although Proposition \ref{Prop-2.21} is splendidly simple it lays the foundation for a more profound result. We use the default convention that $(x)^0=1$ hence, $(\prod\limits_{i=1}^{n}i) (x)^0 = \prod\limits_{i=1}^{n}i$.

\begin{thm}
Consider a graph G with degree string $(d_1)^{t_1}\circ (d_2)^{t_2}\circ (d_3)^{t_3}\circ \ldots \circ (d_l)^{t_l}$, $d_1>d_2>d_3>\ldots>d_l$. Assume that there exists at least one identity subsequence $(d_j)^{t_j}; ~ t_j \ge 2$ with at least two vertices $u,v$ with corresponding degree $d_j$, adjacent. Then, we have
\begin{equation*} 
\begin{cases}
cn^c(G - uv)\ge cn^c(G), & \text{if}~ t^*_j \ge t_{j+1},\\
cn^c(G - uv)< cn^c(G), & \text{if} ~ t^*_j < t_{j+1}.
\end{cases}
\end{equation*} 
with $t_j = t^*_j +2$.
\end{thm}
\begin{proof}
If the string factor $(d_{j+1})^{t_{j+1}}$ does not exists, add the default factor $(d_{j+1})^0$. Now $cn^c(G) = \prod\limits_{i=1}^{l}t_i = (\prod\limits_{i=1,\,i\ne j,j+1}^{l}t_i) t_j t_{j+1}$ or $(\prod\limits_{i=1,\,i\ne j,j+1}^{l}t_i)t_j(d_{j+1})^0$. The compound curling number of the graph $G-uv$ is either $(\prod\limits_{i=1}^{l}t_i) t^*_j (t_{j+1}+2)$ or $(\prod\limits_{i=1}^{l}t_i)t^*_j+2$, where $i\ne j,j+1$. Hence, $cn^c(G-uv) \ge cn^c(G)$ if $t^*_j (t_{j+1}+2) \ge (t^*_j +2) t_{j+1}$, thus if $t^*_j \ge t_{j+1}$. 

\vspace{0.2cm}

\ni We can prove the converse part also using similar arguments.  
\end{proof}

\section{Curling Number of Jaco Graphs}

The notion of Jaco graphs was introduced in \cite{KFW2M}. In this paper, we use  Jaco graphs of  order $1$ which is defined in \cite{KFW2M} as follows.

\begin{defn}{\rm 
A \textit{Jaco graph of order $1$}, denoted by $J_\infty(1)$, is defined as a directed graph with the vertex set $V=V(J_\infty(1)) = \{v_i: i \in \N\}$, and the arc set $A(J_\infty(1)) \subseteq \{(v_i, v_j): i, j \in \N, i< j\}$ and $(v_i,v_ j) \in A(J_\infty(1))$ if and only if $2i - d^-(v_i) \ge j$.}
\end{defn}  

\ni  The Jaco graph $J_\infty(1)$ has the following four fundamental properties.
\begin{enumerate}\itemsep0mm
\item[(i)] $V(J_\infty(1)) = \{v_i:i \in \N\}$, 
\item[(ii)] if $v_j$ is the head of an arc then the tail is always a vertex $v_i, i<j$,
\item[(iii)] if $v_k$, for smallest $k \in \N$ is a tail vertex then all vertices $v_ l, k< l<j$ are tails of arcs to $v_j$, 
\item[(iv)] the degree of vertex $k$ is $d(v_k) = k$.
\end{enumerate} 

The family of finite Jaco graphs are those limited to $n \in \N$ vertices by lobbing off all vertices (and arcs) $v_t, t > n$ and denoted $J_n(1)$. Hence, trivially we have $d(v_i) \le i$ for $i \in \N$.

When defined directed graphs like Jaco graphs are considered the study will apply to the underlying graph. For $J_n(1)$ we denote the underlying graph $J^*_n(1)$. When the context is clear a Jaco graph could mean the directed graph, $J_n(1)$ or the underlying graph, $J^*_n(1)$. We now provide the degree sequences of Jaco graphs, $n \le 26$ in the defined order of vertex labelling (indexing) together with the values $ic(J^*_n(1)), \vartheta(J^*_n(1)), cn(J^*_n(1))$. See the Fisher algorithm in \cite{KFW2M} for the degree sequences.

\begin{ill}\label{Ill-2.24}{\rm 

The degree sequence, number of identity curling subsequences, curling index and curling number and  of first few finite Jaco graphs $J^*_n(1)$ are provided in the following table.}

\begin{table}[h!]
\begin{center}
\begin{tabular}{|c|c|c|c|c|}
\hline
$n$ & degree sequence & $ic$ & $\vartheta$ & $cn$ \\
\hline
1 & (0) & 1 & 1 & 1 \\
\hline
2 & (1,1) & 1 & 1 & 2 \\
\hline
3 & (1,2,1) & 1 & 1 & 2 \\
\hline
4 & (1,2,2,1) & 2 & 4 & 2 \\
\hline
5 & (1,2,3,2,2) & 1 & 1 & 3 \\
\hline
6 & (1,2,3,3,3,2) & 1 & 1 & 3 \\
\hline
7 & (1,2,3,4,4,3,3) & 1 & 1 & 3 \\
\hline
8 & (1,2,3,4,5,4,4,3) & 1 & 1 & 3 \\
\hline
9 & (1,2,3,4,5,5,5,4,3) & 1 & 1 & 3 \\
\hline
10 & (1,2,3,4,5,6,6,5,4,4) & 1 & 1 & 3 \\
\hline
11 & (1,2,3,4,5,6,7,6,5,5,4) & 1 & 1 & 3 \\
\hline
12 & (1,2,3,4,5,6,7,7,6,6,5,4) & 1 & 1 & 3 \\
\hline
13 & (1,2,3,4,5,6,7,8,7,7,6,5,5) & 2 & 4 & 3 \\
\hline
14 & (1,2,3,4,5,6,7,8,8,8,7,6,6,5) & 2 & 4 & 3 \\
\hline
15 & (1,2,3,4,5,6,7,8,9,9,8,7,7,6,6) & 2 & 4 & 3 \\
\hline
16 & (1,2,3,4,5,6,7,8,9,10,9,8,8,7,7,6) & 2 & 4 & 3 \\
\hline
17 & (1,2,3,4,5,6,7,8,9,10,10,9,9,8,8,7,6) & 2 & 4 & 3 \\
\hline
18 & (1,2,3,4,5,6,7,8,9,10,11,10,10,9,9,8,7,7) & 3 & 9 & 3 \\
\hline
19 & (1,2,3,4,5,6,7,8,9,10,11,11,11,10,10,9,8,8,7) & 3 & 9 & 3 \\
\hline
20 & (1,2,3,4,5,6,7,8,9,10,11,12,12,11,11,10,9,9,8,8) & 3 & 9 & 3 \\
\hline
21 & (1,2,3,4,5,6,7,8,9,10,11,12,13,12,12,11,10,10,9,9,8) & 3 & 9 & 3 \\
\hline
22 & (1,2,3,4,5,6,7,8,9,10,11,12,13,13,13,12,11,11,10,10,9,8) & 3 & 9 & 3 \\
\hline
23 & (1,2,3,4,5,6,7,8,9,10,11,12,13,14,14,13,12,12,11,11,10,9,9) & 3 & 9 & 3 \\
\hline
24 & (1,2,3,4,5,6,7,8,9,10,11,12,13,14,15,14,13,13,12,12,11,10,10) & 3 & 9 & 3 \\
\hline
25 & (1,2,3,4,5,6,7,8,9,10,11,12,13,14,15,15,14,14,13,13,12,11,11,10,9) & 3 & 9 & 3 \\
\hline
\end{tabular}
\end{center}
\end{table}
\end{ill}
In the above table $ic, ~ \vartheta$ and $cn$ stands for $ic(J^*_n(1))$, $\vartheta(J^*_n(1))$ and $cn(J^*_n(1))$ respectively.

From Illustration \ref{Ill-2.24}, we see that  any degree sequence of $J^*_n(1), n \ge 5$ can be re-arrange to have one or more 3-entry identity curling subsequence(s) at the tail. For example,  the degree sequence of the Jaco graph $J^*_{19}(1)$ can be re-arranged as $(1,2,3,4,5,6,7,7,9,8,8,8,10,10,10,11,11,11) = (1,2,3,4,5,6,7,7,9)\circ (8,8,8)\circ (10,10,10)\circ (11,11,11)$. We also note that $ic(J^*_n(1)) \ge ic(J^*_m(1))$ if $n > m$. The Fisher algorithm ensures that this will always be the case.

The following theorem establishes the existence and certain characteristics of the identity curling subsequences for a Jaco graph.

\begin{thm}\label{Thm-3.1}
A finite Jaco graph $J^*_n(1), ~ n \ge 5$ has at least one identity curling subsequence and all identity curling subsequences have cardinality $3$. 
\end{thm}
\begin{proof}
Illustration \ref{Ill-2.24} shows that the result holds for $J^*_n(1), 5 \le n \le 26$. It is important to note that besides the possibility of a Jaconian set of cardinality $3$ (see for example $J^*_{22}(1)$) all other vertex degree repeats following are a double repeat. Hence, only after re-arrangement does a vertex degree preceding the Jaconian vertex degree, and equal to the doubly repeated vertex degree cluster to provide a $3$-entry identity curling subsequence.  Assume that it holds for $J^*_k(1), k > 26$, which we know has at least $4$ such distinct double vertex degree repeats. Also assume $J^*_k(1)$ has $l \ge 4$ such distinct double vertex degree repeats following the Jaconian set. Consider any such double vertex degree repeat at say $v_j, v_{j+1}$.

Now, consider $J_{k+1}(1)$. From the definition of a Jaco graph the edges $v_jv_{k+1}$ and $v_{j+1}v_{k+1}$ with perhaps $d(v_j) = j$. Since, the vertex degrees of all vertices $v_{h>j+1}$ are less than $d(v_{j+1})$ no triple vertex degree repeat is possible. This holds for all $l \ge 4$ double degree repeats. Hence, $J_{k+1}(1)$ has at least $l \ge 4$ double degree repeats. By re-arranging the degree sequence, subsequences of triple vertex degree repeats can be clustered. This completes the proof.
\end{proof}

\begin{cor}
For a Jaco graph $J_n(1), n \ge 1$ we have $1 \le cn(J^*_n(1)) \le 3$.
\end{cor}
\begin{proof}
Illustration \ref{Ill-2.24} shows that $cn(J^*_1(1)) = 1$, $cn(J^*_2(1)) = cn(J^*_3(1)) = cn (J^*_4(1)) \\ = 2$ and from Theorem \ref{Thm-3.1}, it follows that $cn(J^*_n(1)) = 3, n \ge 5$.
\end{proof}

\section{Curling Number of Set-Graphs}

In this section, we discuss the curling number of a special class of graphs called set-graphs. Let us first recall the definition of set-graphs as given in \cite{KC2S}. 

\begin{defn}\label{D-SG}{\rm 
\cite{KC2S} Let $A^{(n)} = \{a_1, a_2, a_3, \ldots, a_n\}, n\in \N$ be a non-empty set and the $i$-th $s$-element subset of $A^{(n)}$ be denoted by $A_{s,i}^{(n)}$.  Now consider $\cS = \{ A_{s,i}^{(n)}: A_{s,i}^{(n)} \subseteq A^{(n)}, A_{s,i}^{(n)} \ne \emptyset\}$. The \textit{set-graph} corresponding to set $A^{(n)}$, denoted $G_{A^{(n)}}$, is defined to be the graph with $V(G_{A^{(n)}}) = \{v_{s,i}: A_{s,i}^{(n)} \in \cS\}$ and $E(G_{A^{(n)}}) = \{v_{s,i}v_{t,j}:~ A_{s,i}^{(n)}\cap A_{t,j}^{(n)}\ne \emptyset\}$, where $s\ne t~ \text{or}~ i\ne j$. }
\end{defn} 

\begin{prop}
{\rm \cite{KC2S}} If $G$ is a set-graph, then $G$ has odd number of vertices.
\end{prop}

An important result, which is proved in \cite{KC2S}, on set-graphs that is relevant in this context is given below.

\begin{thm}\label{T-SGDV}
{\rm \cite{KC2S}} The vertices $v_{s,i}, v_{s,j}$ of $G_{A^{(n)}}$, corresponding to subsets $A^{(n)}_{s,i}$ and $A^{(n)}_{s,j}$ in $\cS$ of equal cardinality, of the set-graph $G_{A^{(n)}}$ have the same degree in $G_{A^{(n)}}$. That is, $d_{G_{A^{(n)}}}(v_{s,i}) = d_{G_{A^{(n)}}}(v_{s,j})$.
\end{thm}

Invoking these concepts of set-graphs, let us now determine the curling number and compound curling number of set-graphs.

\begin{thm}\label{T-CNSG}
Let $G=G_{A^{(n)}}$ be a set-graph with respect to a non-empty finite set $A=A^{(n)}$. Then, the curling number of $G$ is $\binom{n}{\lfloor\frac{n}{2}\rfloor}$ and the compound curling number of $G$ is $\prod\limits_{i=1}^n\binom{n}{i}$.
\end{thm}
\begin{proof}
Given that $G=G_{A^{(n)}}$ be a set-graph with respect to a non-empty finite set $A=A^{(n)}$. By Theorem \ref{T-SGDV}, the vertices of $G$ corresponding to the subsets of $A$ with same cardinality, have the same degree. Let $r_i$ be the degree of all vertices of $G$ corresponding to $i$-element subsets of $A$, where $1\le i\le n$.  We know that there are $\binom{n}{i}$ vertices in $G$ which have the same degree $r_i$. Then, the degree sequence of $G$ can be written in the string form as $(r_1)^{\binom{n}{1}}\circ (r_2)^{\binom{n}{2}}\circ (r_3)^{\binom{n}{3}}\circ \ldots \circ (r_n)^{\binom{n}{n}}$. We know that the maximum value among the binomial coefficients $\binom{n}{i};~ 1\le i\le n$ is $\binom{n}{\frac{n}{2}}$ if $n$ is even and $\binom{n}{\frac{n-1}{2}}$ (or $\binom{n}{\frac{n+1}{2}}$). Therefore, the curling number of $G$ is $\binom{n}{\lfloor\frac{n}{2}\rfloor}$. 

Also, the compound curling number of the set-graph $G$ is $\binom{n}{1} \binom{n}{2} \binom{n}{3} \ldots \binom{n}{n}= \prod\limits_{i=1}^n\binom{n}{i}$. This completes the proof.
\end{proof}

It is quite interesting to check the number theoretic properties of the curling number and compound curling number of set-graphs. The compound curling number of a set-graph can be rewritten in terms of factorial and hyperfactorial functions as follows. 

\begin{prop}
The compound curling number of a set-graph $G=G_{A^{(n)}}$  is given by $cn^c(G)=\frac{1}{(n\,!)^{(n+1)}}(\prod\limits_{i=1}^n i^i)^2-1$.
\end{prop}
\begin{proof}
The hyperfactorial function $H(n)$ is defined as $H(n)=(\prod\limits_{i=1}^n i^i)$. Moreover, we have $\prod_{i=0}^{n}\binom{n}{i}=\frac{H^2(n)}{\frac{1}{(n\,!)^{(n+1)}}}$ (see \cite{AS1}). Therefore, we have the compound curling number $cn^c(G)$ is $\frac{1}{(n\,!)^{(n+1)}}(\prod\limits_{i=1}^n i^i)^2-1$.
\end{proof}

\begin{thm}
The compound curling number of a set -graph $G=G_{A^{(n)}}$ is a perfect square if and only if $n$ is an odd integer.
\end{thm}
\begin{proof}
If $n$ is an odd integer, then, by the result $\binom{n}{r}=\binom{n}{n-r}$, we have $cn^c(G)=\binom{n}{1} \binom{n}{2} \ldots \binom{n}{\frac{n-1}{2}}  \binom{n}{\frac{n+1}{2}} \ldots \binom{n}{n-1}\binom{n}{n}=\binom{n}{1} \binom{n}{2} \ldots \binom{n}{\frac{n-1}{2}} \binom{n}{\frac{n-1}{2}} \ldots \binom{n}{1}.1= (\prod\limits_{i=1}^n\binom{n}{\frac{n-1}{2}})^2$.

If $n$ is an even integer, then by using the result, we have $cn^c(G)=(\prod\limits_{i=1}^n\binom{n}{\frac{n}{2}-1})^2 \binom{n}{\frac{n}{2}}$. Since, $\binom{n}{\frac{n}{2}}$ is not a perfect square for any positive integer $n$, $cn^c(G)$ can not be a perfect square in this case. This completes the proof.
\end{proof}

\section{Conclusion}
In this paper, we have introduced the concepts of curling number and compound curling number of a given graph and discussed certain properties of these new parameters for certain standard graphs and digraphs. More problems regarding the curling number and compound curling number of certain other graph classes, graph operations, graph products and graph powers are still to be settled. Some other problems we have identified in this area for further works are the following. 

\begin{prob}{\rm 
Determine the condition for the compound curling number of a spanning subgraph of a given regular graph $G$ is greater than that of $G$.}
\end{prob}

\begin{prob}{\rm 
Characterise the graphs in accordance with its compound curling number and that of its spanning subgraphs.}
\end{prob}

\begin{prob}{\rm 
Verify the existence of non-regular graphs whose curling numbers and compound curling numbers are equal and characterise these graphs if exists.}
\end{prob}

There are more problems in this area which seem to be promising for further investigations. All these facts highlights a wide scope for further studies in this area.

\end{document}